\documentclass[a4paper,12pt,reqno]{amsart}
\usepackage{amssymb}
\usepackage{t1enc}
\usepackage[margin=1in]{geometry}
\usepackage{ifthen}
\usepackage{graphicx}
\usepackage{psfrag}
\usepackage{multirow}
\usepackage{xcolor}
\usepackage[footnotesize]{caption}
\usepackage{ctable}

\newcommand{\abs}[1]{\left| #1 \right|}
\newcommand{\expr}[1]{\left( #1 \right)}

\newcommand{\ceil}[1]{\left\lceil #1 \right\rceil}
\newcommand{\norm}[1]{\left\| #1 \right\|}
\newcommand{\set}[1]{\left\{ #1 \right\}}

\newcommand{\ind}{\mathbf{1}}

\newcommand{\sub}{\subseteq}

\newcommand{\R}{\mathbf{R}}
\newcommand{\Z}{\mathbf{Z}}

\newcommand{\A}{\mathcal{A}}
\newcommand{\E}{\mathcal{E}}

\newcommand{\eps}{\varepsilon}
\newcommand{\ph}{\varphi}

\newcommand{\ignore}[1]{}
\newcommand{\q}[1]{}
\newcommand{\na}{n/a}
\newcommand{\red}[1]{\textcolor{red}{#1}}
\newcommand{\blu}[1]{\textcolor{blue}{#1}}
\newcommand{\BK}{\tmark[1]}
\newcommand{\CS}{\tmark[2]}
\newcommand{\KK}{\tmark[3]}
\newcommand{\ZR}{\tmark[4]}

\newcommand{\formula}[2][nolabel]
{\ifthenelse{\equal{#1}{nolabel}}
 {\begin{align*} #2 \end{align*}}
 {\ifthenelse{\equal{#1}{}}
  {\begin{align} #2 \end{align}}
  {\begin{align} \label{#1} #2 \end{align}}
 }
}

\DeclareMathOperator{\pv}{pv}

\DeclareMathOperator{\real}{Re}

\DeclareMathOperator{\li}{Li}

\theoremstyle{plain}
\newtheorem{theorem}{Theorem}
\newtheorem*{theorem*}{Theorem}
\newtheorem{lemma}{Lemma}

\newtheorem{proposition}{Proposition}

\theoremstyle{definition}

\theoremstyle{remark}
\newtheorem*{notation}{Notation}

\title{Eigenvalues of the fractional Laplace operator in the interval}
\author{Mateusz Kwa{\'s}nicki}
\thanks{Work supported by the Polish Ministry of Science and Higher Education grant no. N~N201 373136}
\address{Institute of Mathematics \\ Polish Academy of Sciences \\ ul. {\'S}niadeckich 8, 00-976 Warszawa, Poland}
\email{m.kwasnicki@impan.pl}
\address{Institute of Mathematics and Computer Science \\ Wroc{\l}aw University of Technology \\ ul. Wybrze{\.z}e Wyspia{\'n}\-skiego 27, 50-370 Wroc{\l}aw, Poland}

\begin{document}

\sloppy

\begin{abstract}
Two-term Weyl-type asymptotic law for the eigenvalues of one-di\-men\-sional fractional Laplace operator $(-\Delta)^{\alpha/2}$ ($\alpha \in (0, 2)$) in the interval $(-1,1)$ is given: the $n$-th eigenvalue is equal to $(n \pi/2 - (2 - \alpha) \pi/8)^\alpha + O(1/n)$. Simplicity of eigenvalues is proved for $\alpha \in [1, 2)$. $L^2$ and $L^\infty$ properties of eigenfunctions are studied. We also give precise numerical bounds for the first few eigenvalues.
\end{abstract}

\maketitle

%
%

\section{Introduction and statement of the result}

Let $D = (-1, 1)$ and $\alpha \in (0, 2)$. Below we study the asymptotic behavior of the eigenvalues of the following spectral problem:
\formula[eq:problem]{
 \expr{-\frac{d^2}{d x^2}}^{\alpha / 2} \ph(x) & = \lambda \ph(x), && x \in D ,
}
where $\ph \in L^2(D)$ is extended to $\R$ by $0$ (for details, see below). It is known that there exist an infinite sequence of eigenvalues $\lambda_n$, $0 < \lambda_1 < \lambda_2 \le \lambda_3 \le ...$, and the corresponding eigenfunctions $\ph_n$ form a complete orthonormal set in $L^2(D)$. The following is the main result of this article.

\begin{theorem}
\label{th}
We have
\formula[eq:asymp]{
 \lambda_n = \expr{\frac{n \pi}{2} - \frac{(2 - \alpha) \pi}{8}}^\alpha + O\expr{\frac{1}{n}} .
}
More precisely, there are absolute constants $C\q{1}, C'\q{1}$ such that
\formula{
 \abs{\lambda_n - \expr{\frac{n \pi}{2} - \frac{(2 - \alpha) \pi}{8}}^\alpha} & \le \frac{C\q{1} (2 - \alpha)}{\sqrt{\alpha}} \, \frac{1}{n}
}
for $n \ge (C'\q{1} / \alpha)^{3 / (2\alpha)}$.
\end{theorem}

The scaling property of the fractional Laplace operator $(-d^2 / d x^2)^{\frac{\alpha}{2}}$ implies that $\lambda_n(k D) = k^{-\alpha} \lambda_n(D)$. Hence, one easily finds the asymptotic formula for any interval.

By following carefully the proof, one can take e.g. $C\q{1} = 30\,000$ and $C'\q{1} = 4\,000$ above. Note that the constant in the error term $O(1/n)$ tends to zero as $\alpha$ approaches $2$, and in the limiting case $\alpha = 2$ (not considered below), we have $\lambda_n = (n \pi / 2)^2$ without an error term. A stronger version of Theorem~\ref{th} for $\alpha = 1$ was proved in~\cite{bib:kkms10}.

\bigskip

The proof of Theorem~\ref{th} is modelled after~\cite{bib:kkms10}. In Section~\ref{sec:aux}, an estimate for the fractional Laplace operator is given. The formula for the eigenfunctions on the half-line from~\cite{bib:k10} is recalled and studied in Section~\ref{sec:hl}. An approximation to eigenfunctions is given in Section~\ref{sec:approx}, Theorem~\ref{th} is proved in Section~\ref{sec:eigv}, and three further properties of eigenfunctions and eigenvalues are studied in Section~\ref{sec:eigf}. Sections~\ref{sec:approx}--\ref{sec:eigf} correspond to Sections~8--10 in~\cite{bib:kkms10}. Proposition~\ref{prop:simple} gives the simplicity of the eigenvalues when $\alpha \in (1, 2)$. The result follows relatively easily from the result for $\alpha = 1$ in~\cite{bib:kkms10}. In Propositions~\ref{prop:square} and~\ref{prop:uniform}, $L^2(D)$ and $L^\infty(D)$ bounds for eigenfunctions are given. Finally, in Section~\ref{sec:num}, numerical estimates of $\lambda_n$ in terms of eigenvalues of large dense matrices are obtained.

\bigskip

First-term Weyl-type asymptotic for $\lambda_n$ was proved by Blumenthal and Getoor in 1959~\cite{bib:bg59}. The best known general estimate for $\lambda_n$ is $\frac{1}{2} (\frac{n \pi}{2})^\alpha \le \lambda_n \le (\frac{n \pi}{2})^\alpha$ due to DeBlassie~\cite{bib:d00} and Chen and Song~\cite{bib:cs05}. The important case of $\alpha = 1$ was studied in detail by several authors, see~\cite{bib:bk04, bib:kkms10} and the references therein. It is known that $(\lambda_n)^{1/\alpha}$ is continuous and increasing in $\alpha \in (0, 2]$, see~\cite{bib:cs05, bib:cs06, bib:d00, bib:dm07}. For a discussion of related results and historical remarks, see e.g.~\cite{bib:bk04, bib:kkms10}. Theorem~\ref{th} is of interest in physics, the asymptotic formula~\eqref{eq:asymp} (without the information about the order of the error term) was supported by numerical experiments in~\cite{bib:zrk07}, and there is a considerable amount of related (mostly numerical) research in physics literature.

Noteworthy, although the values of $C\q{1}$ and $C'\q{1}$ given above are rather large, numerical evidence suggests that the error term in formula~\eqref{eq:asymp} is rather small also for small $n$ in the full range of $\alpha \in (0, 2)$, see Table~\ref{tab} and the estimates in the last section of this article. It is an interesting open problem to prove Theorem~\ref{th} with $C\q{1}$ and $C'\q{1}$ non-exploding as $\alpha$ approaches $0$. This is related to simplicity of eigenvalues $\lambda_n$, conjectured to hold for all $\alpha \in (0, 2)$, proved for $\alpha = 1$ in~\cite{bib:kkms10}, and extended to $\alpha \in [1, 2)$ in Proposition~\ref{prop:simple} in Section~\ref{sec:eigf}.

\begin{ctable}%
[label=tab,botcap,doinside={\scriptsize},caption={Comparison of the approximation $\tilde{\lambda}_n = (\frac{n \pi}{2} - \frac{(2 - \alpha) \pi}{8})^\alpha$ (roman font), and numerical approximations to $\lambda_n$ obtained using the method of~\cite{bib:zrk07} with $5000 \times 5000$ matrices (slanted font).}]%
{|l|rr|rr|rr|}{}%
{
\hline
\multicolumn{1}{|c|}{$\alpha$} & \multicolumn{2}{c|}{$\lambda_1$} & \multicolumn{2}{c|}{$\lambda_2$} & \multicolumn{2}{c|}{$\lambda_3$} \\
\hline
0.01 & 0.998 & \sl 0.997 & 1.009 & \sl 1.009 &  1.014 & \sl  1.014 \\
0.1  & 0.981 & \sl 0.973 & 1.091 & \sl 1.092 &  1.147 & \sl  1.148 \\
0.2  & 0.971 & \sl 0.957 & 1.195 & \sl 1.197 &  1.319 & \sl  1.320 \\
0.5  & 0.991 & \sl 0.970 & 1.598 & \sl 1.601 &  2.029 & \sl  2.031 \\
1    & 1.178 & \sl 1.158 & 2.749 & \sl 2.754 &  4.316 & \sl  4.320 \\
1.5  & 1.611 & \sl 1.597 & 5.055 & \sl 5.059 &  9.592 & \sl  9.597 \\
1.8  & 2.056 & \sl 2.048 & 7.500 & \sl 7.501 & 15.795 & \sl 15.801 \\
1.9  & 2.248 & \sl 2.243 & 8.594 & \sl 8.593 & 18.710 & \sl 18.718 \\
1.99 & 2.444 & \sl 2.442 & 9.733 & \sl 9.729 & 21.820 & \sl 21.829 \\
\hline
}
\end{ctable}

Motivated by the results of~\cite{bib:kkms10} and~\cite{bib:k10}, as well as by Theorem~\ref{th} above, one can conjecture asymptotic law similar to~\eqref{eq:asymp} for eigenvalues on an interval for more general operators $\A = \psi(-d^2/dx^2)$, studied in~\cite{bib:k10}. While such a result for each individual $\psi$ should present no difficulty (under some reasonable assumptions on the growth of $\psi$ at infinity), it is an interesting (and much more dificult) problem to obtain estimates uniform also in $\psi$, for a given class of $\psi$. One important example here is the family of Klein-Gordon square-root operators $\A = \sqrt{m^2 - d^2/dx^2} - m$, with mass $m$ ranging from $0$ to $\infty$. This operator is close to $\sqrt{-d^2/dx^2}$ for small $m$, but when $m$ is large, it more similar to $-d^2/dx^2$.

\bigskip

To give a formal statement of the spectral problem~\eqref{eq:problem}, we recall the definition of the one-dimensional fractional Laplace operator $\A = (-d^2/dx^2)^{\alpha/2}$. It is defined pointwise by the principal value integral, if convergent,
\formula[eq:lap]{
 \A f(x) & = c_\alpha \pv\!\!\int_{-\infty}^\infty \frac{f(x) - f(y)}{|x - y|^{1 + \alpha}} \, dy , && x \in \R ,
}
where
\formula{
 c_\alpha & = \frac{2^\alpha \Gamma(\frac{1 + \alpha}{2})}{\sqrt{\pi} \, |\Gamma(-\frac{\alpha}{2})|} \, ;
}
$\A f(x)$ is convergent if, for example, $f$ is smooth in a neighborhood of $x$ and bounded on $\R$. Note that
\formula[eq:cest]{
 \tfrac{1}{8} \alpha (2 - \alpha) \le c_\alpha & \le \tfrac{1}{2} \alpha (2 - \alpha) .
}
For $f \in C_c^\infty(\R)$, the Fourier transform of $\A f$ is equal to $|\xi|^\alpha \hat{f}(\xi)$, and $\A$ extends to an unbounded self-adjoint operator on $L^2(\R)$. We write $\A_D$ for the operator $\A$ on $D$ with zero exterior condition on $\R \setminus D$. More precisely, for $f \in C_c^\infty(D)$, $\A_D f$ is defined to be the restriction of $\A f$ to $D$. Again, $\A_D$ extends to an unbounded self-adjoint operator on $L^2(D)$.

The operator $-\A$ (on an appropriate domain) is the generator of the one-dimensional symmetric $\alpha$-stable process $X_t$, and $-\A_D$ is the generator of $X_t$ killed upon leaving the interval $D$. This probabilistic interpretation is a primary source of our motivation, but will not be exploited in the sequel.

\begin{notation}
Throughout this article, $C$ denotes an absolute constant (independent of $\alpha$). We will track the dependence of other constants employed below on $\alpha$ to catch their asymptotic behavior as $\alpha \searrow 0$ and $\alpha \nearrow 2$. For brevity, we denote $\beta = 2 - \alpha$.
\end{notation}

%
%

\section{Auxiliary estimates}
\label{sec:aux}

Define, as in~\cite{bib:kkms10}, Appendix~C, an auxiliary function:
\formula[eq:q]{
 q(x) & = \begin{cases}
 0 & \text{for } x \in (-\infty, -\tfrac{1}{3}) , \\
 \tfrac{9}{2} (x + \tfrac{1}{3})^2 & \text{for } x \in (-\tfrac{1}{3}, 0) , \\
 1 - \tfrac{9}{2} (x - \tfrac{1}{3})^2 & \text{for } x \in (0, \tfrac{1}{3}) , \\
 1 & \text{for } x \in (\tfrac{1}{3}, \infty) .
 \end{cases}
}
Note that $q$ is piecewise $C^2$, and $q(x) + q(-x) = 1$. Fix a a piecewise $C^2$ function $f$ on $\R$, and let $g(x) = q(x) f(x)$. Further, we assume that the support of $g$ is compact. Below we estimate $\A g$ on $(-1, 0)$ in a very similar way as in~\cite{bib:kkms10}.

Choose $M$ to be the supremum of $\max(|f(x)|, |f'(x)|, |f''(x)|)$ over $x \in (-\frac{1}{3}, \frac{1}{3})$. Let $I = \int_0^\infty |f(x)| dx$. Then
\formula{
 \abs{g''(x)} & \le \abs{f(x) q''(x)} + 2 \abs{f'(x) q'(x)} + \abs{f''(x) q(x)} \le C\q{2} M.
}
Suppose first that $x \in (-1, -\frac{1}{3})$. Since $g$ vanishes in $(-1, -\frac{1}{3})$, $c_\alpha^{-1} |\A g(x)|$ is bounded above by
\formula{
 \int_{-\frac{1}{3}}^\infty \frac{\abs{g(y)}}{|x - y|^{1 + \alpha}} \, dy & \le M \int_{-\frac{1}{3}}^{\frac{1}{3}} \frac{q(y)}{|x - y|^{1 + \alpha}} dy + \frac{3^{1 + \alpha}}{2^{1 + \alpha}} \int_{\frac{1}{3}}^{\infty} \abs{f(y)} dy \\
 & \le \frac{2^{1 - \alpha} 3^\alpha M_0}{2 - \alpha} + \frac{3^{1 + \alpha} I}{2^{1 + \alpha}} \le \frac{C\q{3} M}{\beta} + C\q{4} I .
}
In the second inequality we used the estimate $q(x) / |x - z|^{1 + \alpha} \le \frac{9}{2}(x + \frac{1}{3})^{1 - \alpha}$. For $x \in (-\frac{1}{3}, 0)$ the principal value integral in the definition of $\A$ can be estimated by splitting it into two parts. By Taylor's expansion of $g$, we have
\formula{
 \abs{\pv\int_{x-\frac{1}{3}}^{x+\frac{1}{3}} \frac{g(x) - g(y)}{|x - y|^{1 + \alpha}} \, dy} & \le \sup \set{\tfrac{1}{2} \abs{g''(y)} \; : \; y \in (x-\tfrac{1}{3}, x+\tfrac{1}{3})} \int_{x - \frac{1}{3}}^{x + \frac{1}{3}} \frac{(x - y)^2}{|x - y|^{1 + \alpha}} \, dy \\
 & \le \frac{\sup \set{\abs{g''(y)} \; : \; y \in (-\tfrac{1}{3}, \tfrac{1}{3})}}{3^{2 - \alpha} (2 - \alpha)} \le \frac{C\q{2} M}{\beta} \, .
}
Here for the second inequality note that $g''(y) = 0$ for $y < -\frac{1}{3}$. Furthermore,
\formula{
 & \abs{\expr{\int_{-\infty}^{x-\frac{1}{3}} + \int_{x+\frac{1}{3}}^\infty} \frac{g(x) - g(y)}{|x - y|^{1 + \alpha}} \, dy} \\
 & \le |g(x)| \expr{\int_{-\infty}^{x-\frac{1}{3}} + \int_{x+\frac{1}{3}}^\infty} \frac{1}{(x - y)^{1 + \alpha}} \, dy + 3^{1 + \alpha} \int_{x+\frac{1}{3}}^\infty |f(y)| dy \le \frac{C\q{5} M}{\alpha} + C\q{6} I .
}
We conclude that
\formula[eq:genest]{
 c_\alpha^{-1} \abs{\A g(x)} & \le \frac{C\q{7} M}{\alpha \beta} + C\q{8} I , && x \in (-1, 0) .
}

%
%

\section{Estimates for half-line}
\label{sec:hl}

The main result of~\cite{bib:k10} is the formula for generalized eigenfunctions for a class of operators on $(0, \infty)$. The case of fractional Laplace operator is studied in~\cite{bib:k10}, Example~1. In particular, the eigenfunction $F_\lambda$ of $\A_{(0, \infty)}$ corresponding to the eigenvalue $\lambda^\alpha$ ($\lambda > 0$) is shown to be $F_\lambda(x) = F(\lambda x) = \sin(\lambda x + \frac{\beta \pi}{8}) - G(\lambda x)$ (recall that $\beta = 2 - \alpha$), where $G$ is a completely monotone function. More precisely, $G$ is the Laplace transform of
\formula[eq:gamma]{
 \gamma(s) & = \frac{\sqrt{2 \alpha} \, \sin(\frac{\alpha \pi}{2})}{2 \pi} \frac{s^\alpha}{1 + s^{2 \alpha} - 2 s^\alpha \cos(\frac{\alpha \pi}{2})} \, \exp \expr{\frac{1}{\pi} \int_0^\infty \frac{1}{1 + r^2} \, \log \frac{1 - r^\alpha s^\alpha}{1 - r^2 s^2} \, dr} .
}
Furthermore, by~\cite{bib:k10}, Lemma~13, we have
\formula[eq:gsupest]{
 G(s) & \le \sin(\tfrac{\beta \pi}{8}) \le C\q{9} \beta ,
}
and
\formula[eq:gint]{
 \int_0^\infty G(s) ds & = \cos(\tfrac{\beta \pi}{8}) - \sqrt{\tfrac{\alpha}{2}} \le C\q{10} \beta .
}
Note that the exponent in~\eqref{eq:gamma} is negative. Furthermore, for $\alpha \in (0, 1]$ we have
\formula{
 1 + s^{2 \alpha} - 2 s^\alpha \cos(\tfrac{\alpha \pi}{2}) \ge (\sin(\tfrac{\alpha \pi}{2}))^2 \ge \alpha^2,
}
while for $\alpha \in (1, 2)$, the left hand side is not less than one. Hence, for all $\alpha \in (0, 2]$,
\formula{
 1 + s^{2 \alpha} - 2 s^\alpha \cos(\tfrac{\alpha \pi}{2}) \ge \min(\alpha^2, 1) \ge \tfrac{\alpha^2}{4} .
}
Finally, $\sin(\frac{\alpha \pi}{2}) \le \alpha (2 - \alpha) = \alpha \beta$. Therefore,
\formula[eq:gammaest]{
 \gamma(s) & \le \frac{2 \sqrt{2 \alpha} \, \beta}{\alpha \pi} \, s^\alpha .
}
By direct integration of the Laplace transform, we obtain that
\formula[eq:gest]{
 G(s) & \le \frac{2 \sqrt{2 \alpha} \, \beta \Gamma(1 + \alpha)}{\alpha \pi} \, s^{-1 - \alpha} \le \frac{C\q{11} \beta}{\sqrt{\alpha}} \, s^{-1 - \alpha}.
}
In a similar manner,~\eqref{eq:gammaest} gives
\formula[eq:dgest]{
 -G'(s) & \le \frac{C\q{12} \beta}{\sqrt{\alpha}} \, s^{-2 - \alpha} , & G''(s) & \le \frac{C\q{13} \beta}{\sqrt{\alpha}} \, s^{-3 - \alpha} .
}

%
%

\section{Approximation to eigenfunctions}
\label{sec:approx}

Let $n$ be a fixed positive integer and $\mu_n = \frac{n \pi}{2} - \frac{\beta \pi}{8}$. Our goal is to show that $\mu_n^\alpha$ is close to $\lambda_n$. Note that $\mu_n \ge \frac{\pi}{4}$ and $\frac{n \pi}{4} \le \mu_n \le \frac{n \pi}{2}$.

We construct approximations $\tilde{\ph}_n$ to eigenfunctions $\ph_n$ by combining shifted eigenfunctions for half-line, $F_{\mu_n}(1 + x)$ and $F_{\mu_n}(1 - x)$, and using the auxiliary function $q$ given above in~\eqref{eq:q} to join them in a sufficiently smooth way. We let
\formula[eq:phitilde]{
 \tilde{\ph}_n(x) & = q(-x) F_{\mu_n}(1 + x) + (-1)^n q(x) F_{\mu_n}(1 - x) .
}

\begin{lemma}
We have
\formula[eq:approx]{
 \| \A_D \tilde{\ph}_n - \mu_n^\alpha \tilde{\ph}_n \|_2 & \le \frac{C\q{14} \beta}{\sqrt{\alpha}} \, \frac{1}{n} \, .
}
\end{lemma}

\begin{proof}
Note that we have
\formula{
 \tilde{\ph}_n(x) & - F_{\mu_n}(1 + x) = -(1 - q(-x)) F_{\mu_n}(1 + x) - (-1)^n q(x) F_{\mu_n}(1 - x) \\
 & = -q(x) (F_{\mu_n}(1 + x) + (-1)^n F_{\mu_n}(1 - x)) \\
 & = q(x) (G_{\mu_n}(1 + x) + (-1)^n G_{\mu_n}(1 - x)) - \sin({\mu_n} (1 + x) + \tfrac{\beta \pi}{8}) \ind_{[1, \infty)}(x) .
}
Denote $h(x) = \sin(\mu_n (1 + x) + \frac{\beta \pi}{8}) \ind_{[1, \infty)}(x)$ and $f(x) = G_{{\mu_n}}(1 + x) + (-1)^n G_{\mu_n}(1 - x)$, $g(x) = q(x) f(x)$. It follows that $\tilde{\ph}_n(x) = F_{\mu_n}(1 + x) + g(x) + h(x)$. For $x \in (-1, 0)$, we have $\A F_{\mu_n}(x) - \mu_n^\alpha F_{\mu_n}(x) = 0$ and $h(x) = 0$. Hence,
\formula[eq:decomposition]{
 |\A \tilde{\ph}_n(x) - \mu_n^\alpha \tilde{\ph}_n(x)| & \le |\A g(x)| + |\A h(x)| + |\mu_n^\alpha g(x)| , && x \in (-1, 0) .
}
We will now estimate each of the summands on the right hand side.

Using convexity of $G$, $-G'$ and $G''$, and estimates~\eqref{eq:gint}, \eqref{eq:gest} and \eqref{eq:dgest}, we obtain that
\formula{
 \sup_{x \in (-\frac{1}{3}, \frac{1}{3})} |f(x)| & \le G(\tfrac{2}{3} \mu_n) + G(\tfrac{4}{3} \mu_n) \le \frac{C\q{15} \beta}{\sqrt{\alpha}} \, \mu_n^{-1 - \alpha} , \\
 \sup_{x \in (-\frac{1}{3}, \frac{1}{3})} |f'(x)| & \le -\mu_n G'(\tfrac{2}{3} \mu_n) - \mu_n G'(\tfrac{4}{3} \mu_n) \le \frac{C\q{16} \beta}{\sqrt{\alpha}} \, \mu_n^{-1 - \alpha} , \\
 \sup_{x \in (-\frac{1}{3}, \frac{1}{3})} |f''(x)| & \le \mu_n^2 G''(\tfrac{2}{3} \mu_n) + \mu_n^2 G''(\tfrac{4}{3} \mu_n) \le \frac{C\q{17} \beta}{\sqrt{\alpha}} \, \mu_n^{-1 - \alpha} , \\
 \int_0^\infty |f(x)| dx & \le \int_0^\infty G_{\mu_n}(1 + x) dx + \int_0^1 G_{\mu_n}(1 - x) dx \\
 & \qquad = \frac{1}{\mu_n} \int_0^\infty G(y) dy \le \frac{C\q{10} \beta}{\mu_n} \, .
}
By~\eqref{eq:genest} and~\eqref{eq:cest},
\formula[eq:est1]{
 |\A g(x)| & \le \frac{C\q{18} \beta}{\sqrt{\alpha}} \, \mu_n^{-1 - \alpha} + C\q{19} \alpha \beta^2 \mu_n^{-1} , && x \in (-1, 0) .
}
Furthermore, $|g(x)| = 0$ for $x \in (-1, -\frac{1}{3})$, and
\formula[eq:est2] {
 |\mu_n^\alpha g(x)| & \le \tfrac{1}{2} \mu_n^\alpha f(x) \le \frac{C\q{20} \beta}{\sqrt{\alpha}} \, \mu_n^{-1}, && x \in (-\tfrac{1}{3}, 0) .
}
Finally, for $x < 0$ we have the following estimate for the oscillatory integral
\formula[eq:est3]{
\begin{split}
 |\A h(x)| & = c_\alpha \abs{\int_1^\infty \frac{\sin(\mu_n (1 + y) + \frac{(2 - \alpha) \pi}{8})}{|x - y|^{1 + \alpha}} \, dy} \\
 & \le c_\alpha \int_1^{1 + \pi / \mu_n} \frac{|\sin(\mu_n (1 + y) + \frac{(2 - \alpha) \pi}{8})|}{|x - 1|^{1 + \alpha}} \, dy \le \frac{c_\alpha}{(1 - x)^{1 + \alpha} \mu_n} \le 2 \alpha \beta \mu_n^{-1} \, .
\end{split}
}
Estimates~\eqref{eq:est1}--\eqref{eq:est3} applied to~\eqref{eq:decomposition} yield that
\formula[eq:est4]{
 |\A \tilde{\ph}_n(z) - \mu_n^\alpha \tilde{\ph}_n(z)| & \le \frac{C\q{21} \beta}{\sqrt{\alpha}} \, \mu_n^{-1} , && z \in (-1, 0) .
}
By symmetry,~\eqref{eq:est4} also holds for $z \in (0, 1)$. Formula~\eqref{eq:approx}, with $\A_D \tilde{\ph}_n$ understood in the pointwise sense, follows. It remains to prove that $\tilde{\ph}_n$ is in the domain of $\A_D$. To this end, we will use the notion of the Green operator $G_D = \A_D^{-1}$. The reader is referred e.g. to~\cite{bib:bkk08} for formal definition and properties of $G_D$.

Since $\A \tilde{\ph}_n$ is bounded on $D$, the function $\tilde{\ph}_n - G_D \A \tilde{\ph_n}$ is a bounded, continuous in $D$, weakly $\alpha$-harmonic function in $D = (-1, 1)$ with zero exterior condition. Such a function is necessarily zero (see~\cite{bib:bb99, bib:d65}). It follows that $\tilde{\ph}_n = G_D \A \tilde{\ph}_n$, and hence $\tilde{\ph}_n$ is in the $L^\infty(D)$ domain of $\A_D$. Since convergence in $L^\infty(D)$ is stronger than the one in $L^2(D)$, the proof is complete.
\end{proof}

\begin{lemma}
\label{lem:norm}
We have
\formula[eq:phinorm]{
 1 - \frac{C\q{22} \beta}{n} & \le \|\tilde{\ph}_n\|_2 \le 1 + \frac{C\q{22} \beta}{n} \, .
}
In particular, there is an absolute constant $K$ such that $\|\tilde{\ph}_n\|_2 \ge \frac{1}{2}$ for $n \ge K$.
\end{lemma}

\begin{proof}
First, note that by direct integration,
\formula{
 \abs{\int_{-1}^1 \expr{\expr{\sin(\mu_n (x + 1) + \tfrac{\beta \pi}{8})}^2 - \tfrac{1}{2}} dx} \le \frac{C\q{23} \beta}{\mu_n} \, .
}
Using~\eqref{eq:phitilde} and~\eqref{eq:gint}, we obtain the lower bound,
\formula{
 \|\tilde{\ph}_n\|_2^2 & \ge \int_{-1}^1 \expr{\sin(\mu_n (x + 1) + \tfrac{\beta \pi}{8})}^2 dx \\
 & \qquad - 4 \int_{-1}^1 \abs{q(-x) G_{\mu_n}(x + 1) \sin(\mu_n (x + 1) + \tfrac{\pi}{8})} dx \ge 1 - \frac{C\q{24} \beta}{\mu_n} \, .
}
In a similar manner,
\formula{
 \|\tilde{\ph}_n\|_2^2 & \le 1 + \frac{C}{\mu_n} + 4 \int_{-1}^1 (G(\mu_n (x + 1)))^2 dx \le 1 + \frac{C\q{25} \beta}{\mu_n} \, ,
}
and the lemma is proved.
\end{proof}

%
%

\section{Proof of Theorem~\ref{th}}
\label{sec:eigv}

Since $\tilde{\ph}_n \in L^2(D)$, we have $\tilde{\ph}_n = \sum_j a_j \ph_j$ for some $a_j$. Moreover, $\|\tilde{\ph}_n\|_2^2 = \sum_j a_j^2$ and $\A_D \tilde{\ph}_n = \sum_j \lambda_j a_j \ph_j$. Let $\lambda_{k(n)}$ be the eigenvalue nearest to $\mu_n^\alpha$. Then
\formula{
 \| \A_D \tilde{\ph}_n - \mu_n^\alpha \tilde{\ph}_n \|_2^2 & = \sum_{j = 1}^\infty (\lambda_j - \mu_n^\alpha)^2 a_j^2 \ge (\lambda_{k(n)} - \mu_n^\alpha)^2 \sum_{j = 1}^\infty a_j^2 = (\lambda_{k(n)} - \mu_n^\alpha)^2 \|\tilde{\ph}_n\|_2^2 .
}
By~\eqref{eq:approx} and Lemma~\ref{lem:norm}, it follows that for $n \ge K$,
\formula[eq:lambda]{
  \abs{\lambda_{k(n)} - \mu_n^\alpha} & \le \frac{C\q{26} \beta}{\sqrt{\alpha}} \, \frac{1}{n} \, .
}
This will enable us to derive a two-term asymptotic formula for $\lambda_j$.

Denote $\eps = \frac{1}{2} \, \frac{\beta \pi}{8}$. We have
\formula[eq:epsest]{
 |(\mu_n \pm \eps)^\alpha - \mu_n^\alpha| & \ge \alpha \eps \min((\mu_n - \eps)^{\alpha - 1}, (\mu_n + \eps)^{\alpha - 1}) \ge C\q{27} \alpha \eps n^{\alpha - 1} .
}
Thus if $|\lambda_{k(n)} - \mu_n^\alpha| \le C\q{27} \alpha \eps n^{\alpha - 1}$, then $\lambda_n \in ((\mu_n - \eps)^\alpha, (\mu_n + \eps)^\alpha)$. By~\eqref{eq:lambda}, this holds true if $n \ge K$ and
\formula{
 n & \ge \expr{\frac{C\q{28} \beta}{\alpha^{3/2} \eps}}^{\frac{1}{\alpha}} = (C'\q{29} \alpha^{-3/2})^{1/\alpha} .
}
Therefore, $L_\alpha = \ceil{(C\q{29} \alpha^{-3/2})^{1/\alpha}}$ (the constant here is chosen so that also $L_\alpha \ge K$) is such that for $n \ge L_\alpha$, each interval $((\mu_n - \eps)^\alpha, (\mu_n + \eps)^\alpha)$ contains an eigenvalue $\lambda_{k(n)}$. In particular $\lambda_{k(n)}$ are distinct for $n \ge L_\alpha$. We claim that there are less than $L_\alpha$ eigenvalues not included in the above class. As in~\cite{bib:kkms10}, the key step will be the trace estimate.

Let $J$ be the set of those $j > 0$ for which $j \ne k(n)$ for all $n \ge L_\alpha$. Denote by $p_t(x - y)$ and $p^D_t(x, y)$ the heat kernels for $\A$ and $\A_D$ respectively; we have $\hat{p}_t(\xi) = \exp(-t |\xi|^\alpha)$. For $t > 0$, we have (see e.g.~\cite{bib:bk09,bib:k98})
\formula{
 \sum_{j = 1}^\infty e^{-\lambda_j t} & = \int_D \sum_{j = 1}^\infty e^{-\lambda_j t} (\ph_j(x))^2 dx = \int_D p^D_t(x, x) dx \le \int_D p_t(0) dx = \frac{2}{\pi} \int_0^\infty e^{-t s^\alpha} ds .
}
In the last step, Fourier inversion formula was used. Hence,
\formula{
 \sum_{j \in J} e^{-\lambda_j t} & = \sum_{j = 1}^\infty e^{-\lambda_j t} - \sum_{n = L_\alpha}^\infty e^{-\lambda_{k(n)} t} \le \frac{2}{\pi} \expr{\int_0^\infty e^{-t s^\alpha} ds - \frac{\pi}{2} \sum_{n = L_\alpha}^\infty e^{-t (\mu_n + \eps)^\alpha}} .
}
The latter series is bounded below by the integral of $e^{-t s^\alpha}$ over $(\mu_{L_\alpha} + \eps, \infty)$. Hence,
\formula{
 \sum_{j \in J} e^{-\lambda_j t} & \le \frac{2}{\pi} \int_0^{\mu_{L_\alpha} + \eps} e^{-t s^\alpha} ds \le \frac{2}{\pi} \, (\mu_{L_\alpha} + \eps) .
}
Taking the limit as $t \searrow 0$, we obtain that
\formula{
 \# J & \le \frac{2}{\pi} \, (\mu_{L_\alpha} + \eps) = L_\alpha - \frac{\beta}{4} + \frac{2 \eps}{\pi} .
}
Since $\eps < \frac{\beta \pi}{8}$, the right hand side is less than $L_\alpha$, and the claim is proved.

By~\cite{bib:cs06, bib:d00}, we have $\lambda_n \le (n \pi/2)^\alpha$. It follows that for all $n < L_\alpha$, we have $\lambda_n < (\mu_{L_\alpha} - \eps)^\alpha$, and so $J = \{1, 2, ..., L_\alpha - 1\}$. We conclude that $k(n) = n$ for all $n \ge L_\alpha$. Theorem~\ref{th} now follows from~\eqref{eq:lambda}.

%
%

\section{Further properties of eigenvalues and eigenfunctions}
\label{sec:eigf}

Int this section three additional properties of $\ph_n$ and $\lambda_n$ are studied. This part is modelled after~\cite{bib:kkms10}, Section~10. A number of open problems is suggested at the end of the section.

\bigskip

\begin{proposition}[cf. Lemma~3 and Corollary~4 in~\cite{bib:kkms10}]
\label{prop:square}
There is a constant $C$ such that,
\formula{
 \|\tilde{\ph}_n - \ph_n\|_2 & \le \frac{C (2 - \alpha)}{n} && \text{when $\alpha \ge 1$,} \\
 \|\tilde{\ph}_n - \ph_n\|_2 & \le \frac{C (2 - \alpha)}{\alpha^{3/2} n^\alpha} && \text{when $\alpha < 1$.}
}
In particular, if $\ph^*_n(x) = \pm\cos(\mu_n x)$ for odd $n$ an $\ph^*_n(x) = \pm\sin(\mu_n x)$ for even $n$, then
\formula{
 \|\ph^*_n - \ph_n\|_2 & \le \frac{C (2 - \alpha)}{\sqrt{n}} && \text{when $\alpha \ge \frac{1}{2}$,} \\
 \|\ph^*_n - \ph_n\|_2 & \le \frac{C (2 - \alpha)}{\alpha^{3/2} n^\alpha} && \text{when $\alpha < \frac{1}{2}$.}
}
for some constant $C$.
\end{proposition}

\begin{proof}
Fix $n \ge L_\alpha + 1$ and $\eps = \frac{1}{2} \, \frac{\beta \pi}{8}$, and write, as in the previous section, $\tilde{\ph}_n = \sum_j a_j \ph_j$. Changing the sign of $\ph_n$ if necessary, we may assume that $a_n > 0$. As in~\eqref{eq:epsest}, for $j \ne n$ we have $|\lambda_j - \mu_n^\alpha| \ge C \alpha n^{\alpha - 1}$. Hence,
\formula{
 \|\A_D \tilde{\ph}_n - \mu_n^\alpha \tilde{\ph}_n\|_2^2 & = \sum_{j = 1}^\infty (\lambda_j - \mu_n^\alpha)^2 a_j^2 \ge C (\alpha n^{\alpha - 1})^2 \sum_{j \ne n} a_j^2 .
}
By~\eqref{eq:approx}, we obtain that
\formula[eq:phiremainder]{
 \|\tilde{\ph}_n - a_n \ph_n\|_2^2 & = \sum_{j \ne n} a_j^2 \le C \expr{\frac{\beta}{\sqrt{\alpha}} \, \frac{1}{n}}^2 \frac{1}{(\alpha n^{\alpha - 1})^2} = C \expr{\frac{\beta}{\alpha^{3/2} n^\alpha}}^2 .
}
Note that
\formula{
 \norm{\tilde{\ph}_n - \|\tilde{\ph}_n\|_2 \ph_n}_2 & \le \norm{\tilde{\ph}_n - a_n \ph_n}_2 + \abs{a_n - \|\tilde{\ph}_n\|_2} .
}
But $|a_n - \|\tilde{\ph}_n\|_2|^2 \le \|\tilde{\ph}_n\|_2^2 - a_n^2 = \|\tilde{\ph}_n - a_n \ph_n\|_2^2$. Hence, using also~\eqref{eq:phinorm}, we obtain that
\formula[eq:phl2]{
 \|\tilde{\ph}_n - \|\tilde{\ph}_n\|_2 \ph_n\|_2 & \le \frac{2 C \beta}{\alpha^{3/2} n^\alpha} \, .
}
Finally, by~\eqref{eq:phiremainder},
\formula{
 \|\tilde{\ph}_n - \ph_n\|_2 & \le \|\tilde{\ph}_n - \|\tilde{\ph}_n\|_2 \ph_n\|_2 + |\|\tilde{\ph}_n\|_2 - 1| \le \frac{2 C \beta}{\alpha^{3/2} n^\alpha} + \frac{C \beta}{n} \, .
}
We have thus proved the first part of the proposition. The second statement is a simple consequence of the first one, the identity $\tilde{\ph}_n(x) = \ph^*_n(x) - G(\mu_n (1 + x)) \pm G(\mu_n (1 - x))$, and the estimate $\|G\|_2^2 \le \|G\|_1 \|G\|_\infty \le C \beta^2$.
\end{proof}

\bigskip

\begin{proposition}[cf. Corollary~5 in~\cite{bib:kkms10}]
\label{prop:uniform}
If $\alpha \ge \frac{1}{2}$, then the eigenfunctions $\ph_n(x)$ are bounded uniformly in $n \ge 1$ and $x \in D$.
\end{proposition}

\begin{proof}
Let $P_t^D = \exp(-t \A_D)$ ($t > 0$) be the heat semigroup for $\A_D$ (or transition semigroup of the symmetric $\alpha$-stable process in $D$), and let $p_t^D(x, y)$ be the corresponding heat kernel (or transition density). It is well known that $p_t^D(x, y) \le p_t(y - x)$, where $p_t(x)$ is the heat kernel for $\A$ and $\hat{p}_t(\xi) = \exp(-t |\xi|^\alpha)$; see e.g.~\cite{bib:bbkrsv09}.

By Cauchy-Schwarz inequality and Plancherel's theorem, we obtain
\formula{
 e^{-\lambda_n t} |\ph_n(x)| & \le |P^D_t(\ph_n - \tilde{\ph}_n)(x)| + |P^D_t \tilde{\ph}_n(x)| \\
 & \le \sqrt{\int_{-\infty}^\infty (p_t(x - y))^2 dy} \, \|\ph_n - \tilde{\ph}_n\|_2 + \|\tilde{\ph}_n\|_\infty \\
 & = \sqrt{\frac{1}{2 \pi} \int_{-\infty}^\infty e^{-2 t |z|^\alpha} dz} \, \|\ph_n - \tilde{\ph}_n\|_2 + \|F\|_\infty \\
 & \le 2 \sqrt{\Gamma(1 + 1/\alpha)} \, (2 t)^{-1 / (2\alpha)} \|\ph_n - \tilde{\ph}_n\|_2 + 2 ,
}
Let $t = 1 / \lambda_n$. Then $e^{-\lambda_n t} = 1 / e$ and $t^{-1 / (2 \alpha)} \le \lambda_n^{1 / (2 \alpha)} \le C \sqrt{n}$. If $n \ge L_\alpha$ and $\alpha \ge \frac{1}{2}$, then also $\|\ph_n - \tilde{\ph}_n\|_2 \le C / \sqrt{n}$, and finally $|\ph_n(x)| \le C$. Since each $\ph_n$ is in $L^\infty(D)$, the proof is complete.
\end{proof}

\begin{proposition}[cf. Theorem~6 in~\cite{bib:kkms10}]
\label{prop:simple}
If $\alpha \ge 1$, then the eigenvalues $\lambda_n$ are simple.
\end{proposition}

\begin{proof}
Let us write $\lambda_{n,\alpha}$ for $\lambda_n$ in this proof. Since $(\lambda_{n,\alpha})^{1/\alpha}$ is increasing in $\alpha$, we have
\formula{
 (\lambda_{n,\alpha})^{1/\alpha} & \le (\lambda_{n,2})^{1/2} = \frac{n \pi}{2} \, .
}
By Theorem~6 in~\cite{bib:kkms10}, for $n \ge 3$ we have
\formula{
 \frac{(n+1) \pi}{2} - \frac{\pi}{8} - \frac{\pi}{10} & < \lambda_{n+1,1} \le (\lambda_{n+1,\alpha})^{1/\alpha} .
}
Therefore, $\lambda_{n,\alpha} < \lambda_{n+1,\alpha}$, except perhaps $n = 1$ or $n = 2$. But a similar argument works also for $n = 1$ and $n = 2$, since by~\cite{bib:bk04} we have
\formula{
 \frac{\pi}{2} & < 2 \le \lambda_{2,1} , && \text{and} & \pi & < 3.83 < \lambda_{3,1} .
}
The proof is complete.
\end{proof}

Numerical experiments suggest that $\ph_n$ are uniformly bounded also for $\alpha < \frac{1}{2}$. Furthermore, it would be interesting to obtain an upper estimate of $\sup_n \|\ph_n\|_\infty$, and in particular, to find its behavior when $\alpha$ approaches $0$. Finally, as stated in the introduction, better bounds for $\lambda_n$ may yield simplicity of eigenvalues also when $\alpha < 1$. 

%
%

\section{Numerical bounds for eigenvalues}
\label{sec:num}

No general \emph{efficient} algorithm giving mathematically correct numerical bounds for $\lambda_n$ is known to the author. For $\alpha = 1$, a satisfactory method (an application of Rayleigh-Ritz and Weinstein-Aronszajn methods) is described in~\cite{bib:kkms10}. For general $\alpha$, even approximation of $\lambda_n$ is difficult: all known methods converge rather slowly, and thus the computation of eigenvalues of very large matrices is required. In this section a method for obtaining a lower bound for $\lambda_n$ is described. It shares the main drawbacks of many related algorithms: compared to the technique applied in~\cite{bib:zrk07}, it converges slowly, and it suffers large errors as $\alpha$ approaches $2$. On the other hand, the method presented below gives mathematically correct lower bounds, and there is no error estimate for the numerical scheme of~\cite{bib:zrk07}. At the end of the section, a somewhat similar method for the upper bound for $\lambda_1$ is given. It gives satisfactory results for large $\alpha$, but deteriorates as $\alpha$ gets close to $0$.

\begin{ctable}%
[label={tab:interval},botcap,notespar,doinside={\scriptsize},caption={Comparison of bounds and approximations to $\lambda_n$. Each cell contains six numbers: lower bound $\lambda_{n,\eps}$ with $\eps = \frac{1}{2500}$, the best lower bound known before, approximation $(\frac{n \pi}{2} - \frac{(2 - \alpha)\pi}{8})^\alpha$, numerical approximation of~\cite{bib:zrk07}, upper bound $\lambda_{1,\eps}^*$, the best upper bound known before. The better estimates are printed in color.}]%
{|@{\hspace{0.25em}}c@{\hspace{0.25em}}*{10}{|@{\hspace{0.25em}}r@{}l@{\hspace{0.25em}}}@{\hspace{0.25em}}|}%
{
\tnote[1]{See~\cite{bib:bk04}.}
\tnote[2]{See~\cite{bib:cs05}.}
\tnote[3]{Combination of~\cite{bib:kkms10} with monotonicity in $\alpha$.}
\tnote[4]{See~\cite{bib:zrk07}.}
}%
{
\hline
\multicolumn{1}{|c|@{\hspace{0.25em}}}{$\alpha$} &
\multicolumn{2}{c|@{\hspace{0.25em}}}{$\lambda_1$} & 
\multicolumn{2}{c|@{\hspace{0.25em}}}{$\lambda_2$} & 
\multicolumn{2}{c|@{\hspace{0.25em}}}{$\lambda_3$} & 
\multicolumn{2}{c|@{\hspace{0.25em}}}{$\lambda_4$} & 
\multicolumn{2}{c|@{\hspace{0.25em}}}{$\lambda_5$} & 
\multicolumn{2}{c|@{\hspace{0.25em}}}{$\lambda_6$} & 
\multicolumn{2}{c|@{\hspace{0.25em}}}{$\lambda_7$} & 
\multicolumn{2}{c|@{\hspace{0.25em}}}{$\lambda_8$} & 
\multicolumn{2}{c|@{\hspace{0.25em}}}{$\lambda_9$} & 
\multicolumn{2}{c|}{$\lambda_{10}$} \\ 
\hline
\multirow{6}{*}{0.01} &     \blu{0.9966} &     &     \blu{1.0086} &     &     \blu{ 1.0137} &     &     \blu{ 1.0171} &     &     \blu{ 1.0196} &     &     \blu{ 1.0217} &     &     \blu{  1.0234} &     &     \blu{  1.0248} &     &     \blu{  1.0261} &     &     \blu{  1.0273} &     \\
                      & \sl      0.9943  & \BK & \sl      0.5057  & \CS & \sl       0.5078  & \CS & \sl       0.5092  & \CS & \sl       0.5104  & \CS & \sl       0.5113  & \CS & \sl        0.5121  & \CS & \sl        0.5128  & \CS & \sl        0.5134  & \CS & \sl        0.5139  & \CS \\
                      &          0.9976  &     &          1.0086  &     &           1.0138  &     &           1.0172  &     &           1.0198  &     &           1.0218  &     &            1.0235  &     &            1.0250  &     &            1.0263  &     &            1.0274  &     \\
                      & \sl      0.9966  & \ZR & \sl      1.0087  & \ZR & \sl       1.0137  & \ZR & \sl       1.0172  & \ZR & \sl       1.0197  & \ZR & \sl       1.0218  & \ZR & \sl        1.0235  & \ZR & \sl        1.0250  & \ZR & \sl        1.0263  & \ZR & \sl        1.0274  & \ZR \\
                      &         13.5210  &     & \na              &     & \na               &     & \na               &     & \na               &     & \na               &     & \na                &     & \na                &     & \na                &     & \na                &     \\
                      & \sl \red{0.9974} & \BK & \sl \red{1.0102} & \KK & \sl \red{ 1.0148} & \KK & \sl \red{ 1.0179} & \KK & \sl \red{ 1.0203} & \KK & \sl \red{ 1.0223} & \KK & \sl \red{  1.0239} & \KK & \sl \red{  1.0254} & \KK & \sl \red{  1.0266} & \KK & \sl \red{  1.0277} & \KK \\ \hline
\multirow{6}{*}{0.1}  &     \blu{0.9724} &     &     \blu{1.0919} &     &     \blu{ 1.1469} &     &     \blu{ 1.1863} &     &     \blu{ 1.2159} &     &     \blu{ 1.2405} &     &     \blu{  1.2611} &     &     \blu{  1.2791} &     &     \blu{  1.2950} &     &     \blu{  1.3094} &     \\
                      & \sl      0.9513  & \BK & \sl      0.5606  & \CS & \sl       0.5838  & \CS & \sl       0.6008  & \CS & \sl       0.6144  & \CS & \sl       0.6257  & \CS & \sl        0.6354  & \CS & \sl        0.6440  & \CS & \sl        0.6516  & \CS & \sl        0.6585  & \CS \\
                      &          0.9809  &     &          1.0913  &     &           1.1477  &     &           1.1867  &     &           1.2167  &     &           1.2412  &     &            1.2620  &     &            1.2802  &     &            1.2962  &     &            1.3107  &     \\
                      & \sl      0.9726  & \ZR & \sl      1.0922  & \ZR & \sl       1.1473  & \ZR & \sl       1.1868  & \ZR & \sl       1.2165  & \ZR & \sl       1.2413  & \ZR & \sl        1.2620  & \ZR & \sl        1.2802  & \ZR & \sl        1.2962  & \ZR & \sl        1.3107  & \ZR \\
                      &          1.8351  &     & \na              &     & \na               &     & \na               &     & \na               &     & \na               &     & \na                &     & \na                &     & \na                &     & \na                &     \\
                      & \sl \red{0.9786} & \BK & \sl \red{1.1067} & \KK & \sl \red{ 1.1575} & \KK & \sl \red{ 1.1941} & \KK & \sl \red{ 1.2226} & \KK & \sl \red{ 1.2462} & \KK & \sl \red{  1.2664} & \KK & \sl \red{  1.2840} & \KK & \sl \red{  1.2997} & \KK & \sl \red{  1.3138} & \KK \\ \hline
\multirow{6}{*}{0.2}  &     \blu{0.9572} &     &     \blu{1.1960} &     &     \blu{ 1.3182} &     &     \blu{ 1.4093} &     &     \blu{ 1.4801} &     &     \blu{ 1.5402} &     &     \blu{  1.5915} &     &     \blu{  1.6373} &     &     \blu{  1.6780} &     &     \blu{  1.7154} &     \\
                      & \sl      0.9181  & \BK & \sl      0.6286  & \CS & \sl       0.6817  & \CS & \sl       0.7221  & \CS & \sl       0.7550  & \CS & \sl       0.7831  & \CS & \sl        0.8076  & \CS & \sl        0.8294  & \CS & \sl        0.8492  & \CS & \sl        0.8673  & \CS \\
                      &          0.9712  &     &          1.1948  &     &           1.3199  &     &           1.4102  &     &           1.4819  &     &           1.5420  &     &            1.5939  &     &            1.6399  &     &            1.6812  &     &            1.7188  &     \\
                      & \sl      0.9575  & \ZR & \sl      1.1965  & \ZR & \sl       1.3191  & \ZR & \sl       1.4105  & \ZR & \sl       1.4817  & \ZR & \sl       1.5421  & \ZR & \sl        1.5938  & \ZR & \sl        1.6400  & \ZR & \sl        1.6811  & \ZR & \sl        1.7188  & \ZR \\
                      &          1.2376  &     & \na              &     & \na               &     & \na               &     & \na               &     & \na               &     & \na                &     & \na                &     & \na                &     & \na                &     \\
                      & \sl \red{0.9675} & \BK & \sl \red{1.2247} & \KK & \sl \red{ 1.3398} & \KK & \sl \red{ 1.4258} & \KK & \sl \red{ 1.4947} & \KK & \sl \red{ 1.5530} & \KK & \sl \red{  1.6036} & \KK & \sl \red{  1.6485} & \KK & \sl \red{  1.6890} & \KK & \sl \red{  1.7260} & \KK \\ \hline
\multirow{6}{*}{0.5}  &     \blu{0.9692} &     &     \blu{1.5991} &     &     \blu{ 2.0247} &     &     \blu{ 2.3809} &     &     \blu{ 2.6862} &     &     \blu{ 2.9618} &     &     \blu{  3.2118} &     &     \blu{  3.4443} &     &     \blu{  3.6608} &     &     \blu{  3.8654} &     \\
                      & \sl      0.8862  & \CS & \sl      0.8862  & \CS & \sl       1.0854  & \CS & \sl       1.2533  & \CS & \sl       1.4012  & \CS & \sl       1.5349  & \CS & \sl        1.6579  & \CS & \sl        1.7724  & \CS & \sl        1.8799  & \CS & \sl        1.9816  & \CS \\
                      &          0.9908  &     &          1.5977  &     &           2.0306  &     &           2.3862  &     &           2.6954  &     &           2.9725  &     &            3.2259  &     &            3.4608  &     &            3.6808  &     &            3.8883  &     \\
                      & \sl      0.9701  & \ZR & \sl      1.6015  & \ZR & \sl       2.0288  & \ZR & \sl       2.3871  & \ZR & \sl       2.6947  & \ZR & \sl       2.9728  & \ZR & \sl        3.2255  & \ZR & \sl        3.4610  & \ZR & \sl        3.6805  & \ZR & \sl        3.8883  & \ZR \\
                      &          1.0002  &     & \na              &     & \na               &     & \na               &     & \na               &     & \na               &     & \na                &     & \na                &     & \na                &     & \na                &     \\
                      & \sl \red{0.9863} & \BK & \sl \red{1.6598} & \KK & \sl \red{ 2.0777} & \KK & \sl \red{ 2.4274} & \KK & \sl \red{ 2.7314} & \KK & \sl \red{ 3.0055} & \KK & \sl \red{  3.2562} & \KK & \sl \red{  3.4892} & \KK & \sl \red{  3.7074} & \KK & \sl \red{  3.9136} & \KK \\ \hline
\multirow{6}{*}{1}    &          1.1516  &     &          2.7343  &     &           4.2756  &     &           5.8236  &     &           7.3584  &     &           8.8919  &     &           10.4166  &     &           11.9382  &     &           13.4528  &     &           14.9636  &     \\
                      & \sl \blu{1.1577} & \KK & \sl \blu{2.7547} & \KK & \sl \blu{ 4.3168} & \KK & \sl \blu{ 5.8921} & \KK & \sl \blu{ 7.4601} & \KK & \sl \blu{ 9.0328} & \KK & \sl \blu{ 10.6022} & \KK & \sl \blu{ 12.1741} & \KK & \sl \blu{ 13.7441} & \KK & \sl \blu{ 15.3155} & \KK \\
                      &          1.1781  &     &          2.7489  &     &           4.3197  &     &           5.8905  &     &           7.4613  &     &           9.0321  &     &           10.6029  &     &           12.1737  &     &           13.7445  &     &           15.3153  &     \\
                      & \sl      1.1577  & \ZR & \sl      2.7545  & \ZR & \sl       4.3164  & \ZR & \sl       5.8916  & \ZR & \sl       7.4594  & \ZR & \sl       9.0319  & \ZR & \sl       10.6012  & \ZR & \sl       12.1729  & \ZR & \sl       13.7427  & \ZR & \sl       15.3140  & \ZR \\
                      &          1.1608  &     & \na              &     & \na               &     & \na               &     & \na               &     & \na               &     & \na                &     & \na                &     & \na                &     & \na                &     \\
                      & \sl \red{1.1578} & \KK & \sl \red{2.7548} & \KK & \sl \red{ 4.3169} & \KK & \sl \red{ 5.8922} & \KK & \sl \red{ 7.4602} & \KK & \sl \red{ 9.0329} & \KK & \sl \red{ 10.6023} & \KK & \sl \red{ 12.1742} & \KK & \sl \red{ 13.7442} & \KK & \sl \red{ 15.3156} & \KK \\ \hline
\multirow{6}{*}{1.5}  &     \blu{1.5139} &     &     \blu{4.7367} &     &           8.8817  &     &          13.7668  &     &          19.2502  &     &          25.2613  &     &           31.7334  &     &           38.6263  &     &           45.8996  &     &           53.5266  &     \\
                      & \sl      1.3293  & \BK & \sl      4.5721  & \KK & \sl \blu{ 8.9689} & \KK & \sl \blu{14.3024} & \KK & \sl \blu{20.3762} & \KK & \sl \blu{27.1479} & \KK & \sl \blu{ 34.5222} & \KK & \sl \blu{ 42.4772} & \KK & \sl \blu{ 50.9536} & \KK & \sl \blu{ 59.9375} & \KK \\
                      &          1.6114  &     &          5.0545  &     &           9.5970  &     &          15.0171  &     &          21.1905  &     &          28.0344  &     &           35.4886  &     &           43.5067  &     &           52.0514  &     &           61.0922  &     \\
                      & \sl      1.5971  & \ZR & \sl      5.0586  & \ZR & \sl       9.5921  & \ZR & \sl      15.0154  & \ZR & \sl      21.1846  & \ZR & \sl      28.0289  & \ZR & \sl       35.4800  & \ZR & \sl       43.4972  & \ZR & \sl       52.0392  & \ZR & \sl       61.0786  & \ZR \\
                      &     \red{1.5989} &     & \na              &     & \na               &     & \na               &     & \na               &     & \na               &     & \na                &     & \na                &     & \na                &     & \na                &     \\
                      & \sl      1.6224  & \BK & \sl \red{5.5684} & \CS & \sl \red{10.2297} & \CS & \sl \red{15.7497} & \CS & \sl \red{22.0108} & \CS & \sl \red{28.9339} & \CS & \sl \red{ 36.4609} & \CS & \sl \red{ 44.5467} & \CS & \sl \red{ 53.1550} & \CS & \sl \red{ 62.2558} & \CS \\ \hline
\multirow{6}{*}{1.8}  &          1.4483  &     &          5.1149  &     &          10.4447  &     &          17.2231  &     &          25.2907  &     &          34.5448  &     &           44.8969  &     &           56.2813  &     &           68.6385  &     &           81.9210  &     \\
                      & \sl \blu{1.6765} & \BK & \sl \blu{6.1965} & \KK & \sl \blu{13.9088} & \KK & \sl \blu{24.3496} & \KK & \sl \blu{37.2347} & \KK & \sl \blu{52.5393} & \KK & \sl \blu{ 70.1002} & \KK & \sl \blu{ 89.9057} & \KK & \sl \blu{111.8432} & \KK & \sl \blu{135.9060} & \KK \\
                      &          2.0555  &     &          7.5003  &     &          15.8014  &     &          26.7233  &     &          40.1148  &     &          55.8658  &     &           73.8905  &     &           94.1188  &     &          116.4923  &     &          140.9605  &     \\
                      & \sl      2.0481  & \ZR & \sl      7.5007  & \ZR & \sl      15.7948  & \ZR & \sl      26.7156  & \ZR & \sl      40.1012  & \ZR & \sl      55.8481  & \ZR & \sl       73.8661  & \ZR & \sl       94.0884  & \ZR & \sl      116.4541  & \ZR & \sl      140.9145  & \ZR \\
                      &     \red{2.0501} &     & \na              &     & \na               &     & \na               &     & \na               &     & \na               &     & \na                &     & \na                &     & \na                &     & \na                &     \\
                      & \sl      2.0777  & \BK & \sl \red{7.8501} & \CS & \sl \red{16.2868} & \CS & \sl \red{27.3353} & \CS & \sl \red{40.8472} & \CS & \sl \red{56.7138} & \CS & \sl \red{ 74.8501} & \CS & \sl \red{ 95.1871} & \CS & \sl \red{117.6664} & \CS & \sl \red{142.2381} & \CS \\ \hline
\multirow{6}{*}{1.9}  &          1.0353  &     &          3.7704  &     &           7.8734  &     &          13.1989  &     &          19.6379  &     &          27.1159  &     &           35.5691  &     &           44.9481  &     &           55.2082  &     &           66.3127  &     \\
                      & \sl \blu{1.8273} & \BK & \sl \blu{6.8573} & \KK & \sl \blu{16.0993} & \KK & \sl \blu{29.0750} & \KK & \sl \blu{45.5221} & \KK & \sl \blu{65.4737} & \KK & \sl \blu{ 88.7686} & \KK & \sl \blu{115.4333} & \KK & \sl \blu{145.3521} & \KK & \sl \blu{178.5468} & \KK \\
                      &          2.2477  &     &          8.5942  &     &          18.7177  &     &          32.4615  &     &          49.7204  &     &          70.4157  &     &           94.4848  &     &          121.8754  &     &          152.5433  &     &          186.4500  &     \\
                      & \sl      2.2432  & \ZR & \sl      8.5926  & \ZR & \sl      18.7101  & \ZR & \sl      32.4503  & \ZR & \sl      49.7021  & \ZR & \sl      70.3905  & \ZR & \sl       94.4503  & \ZR & \sl      121.8313  & \ZR & \sl      152.4878  & \ZR & \sl      186.3822  & \ZR \\
                      &     \red{2.2455} &     & \na              &     & \na               &     & \na               &     & \na               &     & \na               &     & \na                &     & \na                &     & \na                &     & \na                &     \\
                      & \sl      2.2748  & \BK & \sl \red{8.8021} & \CS & \sl \red{19.0178} & \CS & \sl \red{32.8505} & \CS & \sl \red{50.1962} & \CS & \sl \red{70.9766} & \CS & \sl \red{ 95.1293} & \CS & \sl \red{122.6024} & \CS & \sl \red{153.3517} & \CS & \sl \red{187.3389} & \CS \\ \hline
\multirow{6}{*}{1.99} &          0.1474  &     &          0.5494  &     &           1.1671  &     &           1.9816  &     &           2.9788  &     &           4.1482  &     &            5.4811  &     &            6.9705  &     &            8.6101  &     &           10.3944  &     \\
                      & \sl \blu{1.9816} & \BK & \sl \blu{7.5121} & \KK & \sl \blu{18.3642} & \KK & \sl \blu{34.1070} & \KK & \sl \blu{54.5469} & \KK & \sl \blu{79.8163} & \KK & \sl \blu{109.7856} & \KK & \sl \blu{144.5508} & \KK & \sl \blu{184.0144} & \KK & \sl \blu{228.2517} & \KK \\
                      &          2.4441  &     &          9.7330  &     &          21.8288  &     &          38.7113  &     &          60.3666  &     &          86.7839  &     &          117.9546  &     &          153.8713  &     &          194.5275  &     &          239.9178  &     \\
                      & \sl      2.4427  & \ZR & \sl      9.7293  & \ZR & \sl      21.8200  & \ZR & \sl      38.6960  & \ZR & \sl      60.3426  & \ZR & \sl      86.7495  & \ZR & \sl      117.9077  & \ZR & \sl      153.8100  & \ZR & \sl      194.4500  & \ZR & \sl      239.8220  & \ZR \\
                      &     \red{2.4452} &     & \na              &     & \na               &     & \na               &     & \na               &     & \na               &     & \na                &     & \na                &     & \na                &     & \na                &     \\
                      & \sl      2.4563  & \CS & \sl \red{9.7573} & \CS & \sl \red{21.8651} & \CS & \sl \red{38.7595} & \CS & \sl \red{60.4267} & \CS & \sl \red{86.8560} & \CS & \sl \red{118.0385} & \CS & \sl \red{153.9670} & \CS & \sl \red{194.6351} & \CS & \sl \red{240.0373} & \CS \\ \hline
}
\end{ctable}

It should be pointed out that in many particular cases ($\alpha$ close to $2$ or $n$ large), the bound $\frac{1}{2} (\frac{n \pi}{2})^\alpha \le \lambda_n \le (\frac{n \pi}{2})^\alpha$ of~\cite{bib:cs05, bib:d00} is sharper than the estimates obtained below, unless extremely large matrices are used. Also, good numerical estimates of $\lambda_n$ are available for $\alpha = 1$ due to~\cite{bib:kkms10}, By the monotonicity of $(\lambda_n)^{1/\alpha}$ in $\alpha$, this gives a lower bound for $\lambda_n$ when $\alpha \in (1, 2)$ and an upper bound for $\alpha \in (0, 1)$. Finally, a good estimate of $\lambda_1$ can be found in~\cite{bib:bk04}. For a comparison of the above, see Table~\ref{tab:interval}.

Our method for the lower bound works for fractional Laplace operator in an arbitrary bounded open set $D \sub \R^d$ (in fact, it can be easily extended to more general pseudo-differential operators, or L{\'e}vy processes). Fix $\eps > 0$ and let $\set{I_k : k \in \Z^d}$ be the partition of $\R^d$ into cubes $I_k = \prod_{j = 1}^d [k_j \eps, (k_j + 1)\eps]$, $k \in \Z^d$. Let $K_\eps \sub \Z^d$ be the set of those $k \in \Z^d$ for which $I_k$ intersects $D$, and let $D_\eps$ be the interior of $\bigcup_{k \in K_\eps} I_k$. Note that $D \sub D_\eps$.

The definition of $\A = (-\Delta)^{\alpha/2}$ in higher dimension is similar to~\eqref{eq:lap}: for smooth bounded functions we have
\formula{
 \A f(x) & = c_{d,\alpha} \pv\!\!\int_{\R^d} \frac{f(x) - f(y)}{|x - y|^{d + \alpha}} \, dy , && x \in \R^d ,
}
where $c_{d,\alpha} = 2^\alpha \Gamma((d + \alpha) / 2) / (\pi^{d/2} |\Gamma(-\frac{\alpha}{2})|)$. Fractional Laplace operator in $D$ with zero exterior condition, denoted $\A_D$, is defined as in dimension one. Below we denote by $\lambda_n$ the eigenvalues of $\A_D$. By domain monotonicity of $\lambda_n$, the eigenvalues for $D$ are not less than than the eigenvalues of its superset $D_\eps$. For notational convenience, we assume that $D = D_\eps$.

The Dirichlet form $\E(f, f)$ corresponding to $\A_D$ is given by
\formula{
 \E(f, f) & = \frac{c_{d,\alpha}}{2} \int_{\R^d} \int_{\R^d} \frac{(f(x) - f(y))^2}{|x - y|^{d + \alpha}} \, dx dy , && f \in L^2(D) .
}
As usual, $f \in L^2(D)$ is extended to $\R^d$ so that $f(x) = 0$ for $x \in \R^d \setminus D$. For $k \in \Z^d$, denote
\formula{
 \|k\| & = \sqrt{\sum_{j = 1}^d (|k_j| + 1)^2} .
}
When $x \in I_k$, $y \in I_l$, $k, l \in \Z^d$, we have $|x - y| \le \eps \|k - l\|$. We define
\formula{
 \nu_k & = \|k\|^{-d - \alpha} \, , & \bar{\nu} & = \sum_{k \in \Z^d} \nu_k ,
}
and
\formula{
 \E_\eps(f, f) & = \frac{c_{d,\alpha} \eps^{-d - \alpha}}{2} \sum_{k, l \in \Z^d} \nu_{k - l} \int_{I_k} \int_{I_l} (f(x) - f(y))^2 dx dy .
}
Clearly, $\E_\eps(f, f) \le \E(f, f)$. By Rayleigh-Ritz variational principle, the eigenvalues $\lambda_n$ are bounded below by the sequence $\lambda_{n,\eps}$ of eigenvalues of the operator corresponding to the Dirichlet form $\E_\eps$. More precisely, $\lambda_{n,\eps}$ are defined in the usual way,
\formula{
 \lambda_{n,\eps} & = \inf \set{\sup \set{\E_\eps(f, f) : f \in L, \, \|f\|_2 = 1} : L < L^2(D), \, \dim L = n} .
}
Here `$L < L^2(D)$' means that $L$ is a linear subspace of $L^2(D)$.

For $f \in L^2(D)$ and $k \in \Z^d$, let $f_k = \eps^{-d} \int_{I_k} f(x) dx$, and define $f^*$ to be equal to $f_k$ on $I_k$. Hence $f^* \in L^2(D)$ is the orthogonal projection of $f$ onto the space of functions constant on each $I_k$, and $\int_{I_k} f^*(x) dx = \int_{I_k} f(x) dx$. In particular, $\|f\|_2^2 = \|f^*\|_2^2 + \|f - f^*\|_2^2$. Furthermore,
\formula{
 \E_\eps(f, f) & = \frac{c_{d,\alpha} \eps^{-d - \alpha}}{2} \sum_{k, l \in \Z^d} \nu_{k - l} \int_{I_k} \int_{I_l} \expr{(f(x))^2 - 2 f(x) f(y) + (f(y))^2} dx dy \\
 & = c_{d,\alpha} \eps^{-\alpha} \expr{\bar{\nu} ||f||_2^2 - \eps^d \sum_{k, l \in \Z^d} \nu_{k - l} f_k f_l} .
}
Comparing this with a similar formula for $f^*$, we obtain that
\formula{
 \E_\eps(f, f) & = \E_\eps(f^*, f^*) + c_{d,\alpha} \eps^{-\alpha} \bar{\nu} ||f - f^*||_2^2
}
This shows that the two orthogonal subspaces, $\{ f \in L^2(D) : f^* = 0 \}$ and $\{ f \in L^2(D) : f^* = f \}$, are invariant under the action of the operator corresponding to $\E_\eps$. The former subspace is in fact its eigenspace, corresponding to the eigenvalue $c_{d,\alpha} \eps^{-\alpha}$. The latter one is finite-dimensional, and the action of $\E_\eps$ in the basis of normalized indicators of $I_k$, $k \in K_\eps$, is given by the following matrix $V$: if $\kappa$ be a bijection between $\{1, 2, ..., |K_\eps|\}$ and $K_\eps$, then $V_{p,q} = c_{d,\alpha} \eps^{-\alpha} (\delta_{p,q} \nu^* - \nu_{\kappa(p) - \kappa(q)})$.

We conclude that the sequence $\lambda_{n,\eps}$ starts with those eigenvalues of the matrix $V$ which are less than $c_{d,\alpha} \eps^{-\alpha} \bar{\nu}$, which are followed by the constant $c_{d,\alpha} \eps^{-\alpha} \bar{\nu}$. This gives the lower bound for the eigenvalues $\lambda_n$ for an arbitrary open bounded set $D$. Note that replacing $\bar{\nu}$ be a smaller number gives smaller lower bounds $\lambda_{n,\eps}$, hence the series defining $\bar{\nu}$ should be approximated from below.

When $D = (-1, 1) \sub \R$ and $\eps = \frac{2}{N}$, then $\bar{\nu} = 2 \zeta(1 + \alpha) - 1$, where $\zeta$ is the Riemann zeta function. Furthermore, in this case $V$ is a Toeplitz matrix with the symbol
\formula{
 \frac{2 c_{d,\alpha}}{\eps^\alpha} \expr{\zeta(1 + \alpha) - \sum_{k = 0}^\infty \frac{\cos(k x)}{(1 + k)^{1 + \alpha}}} & = \frac{2 c_{d,\alpha}}{\eps^\alpha} \expr{\zeta(1 + \alpha) - \real \expr{\frac{\li_{1+\alpha}(e^{i x})}{e^{i x}}}} \\
 & \hspace*{-4em} = \frac{2 c_{d,\alpha}}{\eps^\alpha} \expr{\zeta(1 + \alpha) - \frac{1}{1 + \alpha} \int_0^\infty \frac{t^\alpha (e^t - \cos x)}{e^{2t} - 2 e^t \cos x + 1} \, dt}.
}
The right hand side is easily checked to be increasing in $x \in [0, \pi]$, and so it attains its maximum for $x = \pi$. The symbol of $V$, and hence the eigenvalues of $V$, are therefore bounded above by $2 c_{d,\alpha} \eps^{-\alpha} (\zeta(1 + \alpha) - \li_{1+\alpha}(-1)) = 2^{1-\alpha} c_{d,\alpha} \eps^{-\alpha} \zeta(1 + \alpha) \le c_{d,\alpha} \eps^{-\alpha} \bar{\nu}$. It follows that all $N$ eigenvalues of $V$ are included in the sequence $\lambda_{n,\eps}$.

In higher dimensions, $\bar{\nu}$ can only be computed by approximating numerically a $d$-dimensional infinite series. We summarize the results of this section in the following two results.

\begin{proposition}
Let $D = (-1, 1)$, $N > 0$ and $\eps = 2 / N$. Let $V$ be a $N \times N$ Toeplitz matrix with entries
\formula{
 V_{p,q} & = -\frac{c_\alpha}{\eps^\alpha} \, \frac{1}{(|p - q| + 1)^{d + \alpha}} \, , && p, q = 1, 2, ..., N, \; p \ne q; \\
 V_{p,p} & = \frac{2 c_\alpha (\zeta(1 + \alpha) - 1)}{\eps^\alpha} \, , && p = 1, 2, ..., N .
}
Define $\lambda_{n,\eps}$ to be the $n$-th smallest eigenvalue of $V$ when $n \le N$, and $\lambda_{n,\eps} = c_\alpha \eps^{-\alpha} (2 \zeta(1 + \alpha) - 1)$ otherwise. Then the eigenvalues $\lambda_n$ of $\A_D$ satisfy $\lambda_n \ge \lambda_{n,\eps}$. 
\end{proposition}

\begin{proposition}
Let $D \sub \R^d$ be an open set in $\R^d$, and let $\eps > 0$. Let $K_\eps$ be the set of those $k \in \Z^d$ for which $D \cap \prod_{j = 1}^d [k_j \eps, (k_j + 1) \eps]$ is nonempty, and let $\kappa : \{1, 2, ..., |K_\eps|\} \to K_\eps$ be the enumeration of elements of $K_\eps$. Finally, let
\formula{
 \bar{\nu} & = \sum_{k \in \Z^d} \|k\|^{-d - \alpha} , && \text{where} & \|k\| & = \sqrt{\sum_{j = 1}^d (|k_j| + 1)^2} .
}
Define a $|K_\eps| \times |K_\eps|$ matrix $V$ with entries
\formula{
 V_{p,q} & = -\frac{c_{d,\alpha}}{\eps^\alpha} \, \|\kappa(p) - \kappa(q)\|^{-d - \alpha} \, , && p, q = 1, 2, ..., |K_\eps|, \; p \ne q; \\
 V_{p,p} & = \frac{c_{d,\alpha}}{\eps^\alpha} \, \expr{\bar{\nu} - d^{-(d + \alpha)/2}} , && p = 1, 2, ..., N .
}
Let $\lambda_{n,\eps}$ be the $n$-th smallest eigenvalue of $V$ if $n \le N$ and this eigenvalue does not exceed $c_{d,\alpha} \eps^{-\alpha} \bar{\nu}$, and $\lambda_{n,\eps} = c_{d,\alpha} \eps^{-\alpha} \bar{\nu}$ otherwise. Then the eigenvalues $\lambda_n$ of $\A_D$ satisfy $\lambda_n \ge \lambda_{n,\eps}$.
\end{proposition}

The lower bounds $\lambda_{n,\eps}$ for the interval $D = (-1, 1)$ are presented in Table~\ref{tab:interval} above. In higher dimensions, the complexity of computations increases dramatically. For example, a unit disk $B(0, 1)$ or a square $[-1, 1]^2$ with $\eps = \frac{1}{25}$ require handling matrices larger than $2000 \times 2000$. Some results for these two cases are given in Tables~\ref{tab:square} and~\ref{tab:disk}.

\begin{ctable}%
[label=tab:square,botcap,notespar,doinside={\scriptsize},caption={Comparison of estimates of $\lambda_n$ for a square $[-1, 1]^2$. LB and UB mean lower bounds and upper bounds respectively. Estimates of this section are given in roman font, best numerical estimates known before are typeset in slanted font. Better estimates are printed in color.}]%
{|l|*{2}{@{\hspace{0.5em}}rr@{}l@{\hspace{0.5em}}|@{\hspace{0.5em}}r@{}l@{\hspace{0.5em}}|}}%
{
\tnote[1]{See~\cite{bib:cs05}.}
}%
{
\hline
\multicolumn{1}{|c|@{\hspace{0.5em}}}{$\alpha$} & 
\multicolumn{3}{c|@{\hspace{0.5em}}}{$\lambda_1$ (LB)} &
\multicolumn{2}{c|@{\hspace{0.5em}}}{$\lambda_1$ (UB)} &
\multicolumn{3}{c|@{\hspace{0.5em}}}{$\lambda_2$ (LB)} &
\multicolumn{2}{c|}{$\lambda_2$ (UB)} \\
\hline
0.1 & \blu{1.0308} & \sl      0.5230  & \BK & \sl \red{1.0462} & \BK & \blu{1.0880} & \sl      0.5415  & \BK & \sl \red{1.0831} & \BK \\
0.2 & \blu{1.0506} & \sl      0.5472  & \BK & \sl \red{1.0946} & \BK & \blu{1.1691} & \sl      0.5865  & \BK & \sl \red{1.1731} & \BK \\
0.5 & \blu{1.1587} & \sl      0.6266  & \BK & \sl \red{1.2534} & \BK & \blu{1.4908} & \sl      0.7452  & \BK & \sl \red{1.4905} & \BK \\
1   & \blu{1.3844} & \sl      0.7853  & \BK & \sl \red{1.5708} & \BK & \blu{2.1807} & \sl      1.1107  & \BK & \sl \red{2.2215} & \BK \\
1.5 & \blu{1.4135} & \sl      0.9843  & \BK & \sl \red{1.9688} & \BK & \blu{2.6029} & \sl      1.6554  & \BK & \sl \red{3.3110} & \BK \\
1.8 &      0.9167  & \sl \blu{1.1271} & \BK & \sl \red{2.2544} & \BK &      1.8164  & \sl \blu{2.1033} & \BK & \sl \red{4.2068} & \BK \\
1.9 &      0.5427  & \sl \blu{1.1792} & \BK & \sl \red{2.3585} & \BK &      1.0984  & \sl \blu{2.2781} & \BK & \sl \red{4.5563} & \BK \\
\hline
}
\end{ctable}

\begin{ctable}%
[label=tab:disk,botcap,notespar,doinside={\scriptsize},caption={Comparison of estimates of $\lambda_n$ for a unit disk. LB and UB mean lower bounds and upper bounds respectively. Estimates of this section are given in roman font, best numerical estimates known before are typeset in slanted font. Better estimates are printed in color.}]%
{|l|*{3}{@{\hspace{0.5em}}rr@{}l@{\hspace{0.5em}}|}@{\hspace{0.5em}}r@{}l@{\hspace{0.5em}}|}%
{
\tnote[1]{See~\cite{bib:bk04}.}
\tnote[2]{See~\cite{bib:cs05}.}
}%
{
\hline
\multicolumn{1}{|c|@{\hspace{0.5em}}}{$\alpha$} & 
\multicolumn{3}{c|@{\hspace{0.5em}}}{$\lambda_1$ (LB)} &
\multicolumn{3}{c|@{\hspace{0.5em}}}{$\lambda_1$ (UB)} &
\multicolumn{3}{c|@{\hspace{0.5em}}}{$\lambda_2$ (LB)} &
\multicolumn{2}{c|}{$\lambda_2$ (UB)} \\
\hline
0.1 & \blu{1.0381} & \sl      1.0157  & \BK & 6.6198 & \sl \red{1.0641} & \BK & \blu{1.0953} & \sl      0.5718  & \CS & \sl \red{ 1.1609} & \CS \\
0.2 & \blu{1.0655} & \sl      1.0396  & \BK & 3.8878 & \sl \red{1.1342} & \BK & \blu{1.1849} & \sl      0.6541  & \CS & \sl \red{ 1.3476} & \CS \\
0.5 & \blu{1.1986} & \sl      1.1618  & \BK & 2.5081 & \sl \red{1.3943} & \BK & \blu{1.5404} & \sl      0.9787  & \CS & \sl \red{ 2.1079} & \CS \\
1   &      1.4734  & \sl \blu{1.5707} & \BK & 2.7588 & \sl \red{2.0944} & \BK & \blu{2.3201} & \sl      1.9158  & \CS & \sl \red{ 4.4429} & \CS \\
1.5 &      1.5387  & \sl \blu{2.3891} & \BK & 4.0668 & \sl \red{3.4131} & \BK &      2.8379  & \sl \blu{3.7502} & \CS & \sl \red{ 9.3648} & \CS \\
1.8 &      1.0087  & \sl \blu{3.2210} & \BK & 5.5014 & \sl \red{4.7468} & \BK &      2.0045  & \sl \blu{5.6114} & \CS & \sl \red{14.6487} & \CS \\
1.9 &      0.5990  & \sl \blu{3.5834} & \BK & 6.1369 & \sl \red{5.2974} & \CS &      1.2165  & \sl \blu{6.4182} & \CS & \sl \red{17.0045} & \CS \\\hline
}
\end{ctable}

\bigskip

In principle, the upper bound is much more difficult. The above approach can be modified to give an upper bound for $\lambda_1$ whenever the Green function for $D$ can be computed. For the fractional Laplace operator, this is the case when $D$ is a ball. By a scaling property, it is enough to consider $D = B(0, 1)$.

Let $G_D$ be the Green function of $D$, $G_D(x, y) = \int_0^\infty p^D_t(x, y) dt$, where $p^D_t$ is the heat kernel for $\A_D$ (see the proof of Proposition~\ref{prop:uniform}). The Green function is the kernel of $\A_D^{-1}$. M.~Riesz proved that
\formula{
 G_D(x, y) & = \frac{\Gamma(\frac{d}{2}) |x - y|^{\alpha - d}}{2^\alpha \pi^{d/2} (\Gamma(\frac{\alpha}{2}))^2} \int_0^{\frac{(1 - x^2)(1 - y^2)}{|x - y|^2}} \frac{s^{\alpha/2 - 1}}{(1 + s)^{d/2}} \, ds \\
 & = \frac{\Gamma(\frac{d}{2}) (1 - x^2)^{\alpha / 2} (1 - y^2)^{\alpha / 2}}{2^\alpha \pi^{d/2} \Gamma(\frac{\alpha}{2}) \Gamma(1 + \frac{\alpha}{2}) |x - y|^{d}} \, {_2 F_1}\expr{\frac{\alpha}{2}, \, \frac{d}{2}; \, 1 + \frac{\alpha}{2}; \, -\frac{(1 - x^2)(1 - y^2)}{|x - y|^2}} .
}
Since the eigenvalues of $\A_D^{-1}$ are $\lambda_n^{-1}$, we have
\formula{
 \frac{1}{\lambda_1} & = \sup \set{\int_D \int_D G_D(x, y) f(x) f(y) dx dy : f \in L^2(D) , \, \|f\|_2 = 1} .
}
Since $G_D(x, y)$ is nonnegative, we may restrict the supremum to nonnegative functions only. Hence, whenever $g(x, y) \le G_D(x, y)$, we have
\formula{
 \lambda_1 & \le \expr{\sup \set{\int_D \int_D g(x, y) f(x) f(y) dx dy : f \in L^2(D) , \, \|f\|_2 = 1}}^{-1} .
}
For $k, l \in \Z^d$, let $g_{k,l}$ be the infimum of $G_D(u, v)$ over $u \in I_k$ and $v \in I_l$. When $x \in I_k$, $y \in I_l$, we choose $g(x, y) = g_{k,l}$. Hence, $\lambda_1$ is bounded above by $\lambda_{1,\eps}^*$, the reciprocal of the largest eigenvalue of the matrix $U$ with entries $U_{i,j} = \eps^d g_{\kappa(i), \kappa(j)}$.

The results for $D = (-1, 1) \sub \R$ and some values of $\alpha$ are given in Table~\ref{tab:interval}. Estimates for the unit disk and the square $[-1, 1]^2$ are given in Tables~\ref{tab:square} and~\ref{tab:disk}. Noteworthy, for the unit disk and $\eps = \frac{1}{25}$, the estimate $\lambda_{1,\eps}^*$ is worse than the one obtained in~\cite{bib:bk04} using analytical methods.

%
%

\subsection*{Acknowledgments} I would like to thank Krzysztof Bogdan and Tadeusz Kulczycki for helpful discussion and valuable suggestions.

%
%

%
%

\end{document}